\documentclass{amsart}

\usepackage{times}
\usepackage{amsfonts}
\usepackage{amssymb}
\usepackage[british]{babel}
\usepackage[all,ps]{xy}

\renewcommand{\a}{\mathbf{a}}

\newcommand{\B}{\mathcal{B}}

\newcommand{\D}{\mathcal{D}}
\newcommand{\Ha}{\mathcal{H}}
\newcommand{\Sa}{\mathcal{S}}
\newcommand{\J}{\mathcal{J}}
\newcommand{\sL}{{\scriptscriptstyle L}}
\newcommand{\LL}{\mathcal{L}}

\newcommand{\Q}{\mathbb{Q}}
\newcommand{\sR}{{\scriptscriptstyle R}}
\newcommand{\sLR}{{\scriptscriptstyle LR}}

\newcommand{\Z}{\mathbb{Z}}
\newcommand{\id}{\mathrm{id}}
\newcommand{\LC}{\mathop{\mathrm{LC}}\nolimits}
\newcommand{\RC}{\mathop{\mathrm{RC}}\nolimits}

\newcommand{\Irr}{\mathop{\mathrm{Irr}}}

\renewcommand{\u}{\underline}

\newcommand{\thu}[1]{\theta_{\smash{\u{#1}}}}
\newcommand{\bas}[2]{\thu{#1}^{\phantom{\vee}} \thu{#2}^\vee}
\newcommand{\cbas}[2]{c_{\u{#2}}^{-1} \thu{#1}^{\phantom{\vee}} \thu{#2}^\vee}

\theoremstyle{changebreak}

\newtheorem{Theo}[subsection]{Theorem}
\newtheorem{Lemma}[subsection]{Lemma}
\newtheorem{Prop}[subsection]{Proposition}
\newtheorem{Cor}[subsection]{Corollary}
\newtheorem{Conj}[subsection]{Conjecture}

\theoremstyle{definition}
\newtheorem{Def}[subsection]{Definition}
\newtheorem{DefProp}[subsection]{Definition/Proposition}

\theoremstyle{remark}
\newtheorem{remark}[subsection]{Remark}

\numberwithin{equation}{section}

\begin{document}

\title{A
new construction of the asymptotic algebra associated to the $q$-Schur
algebra}

\author{Olivier Brunat}
\address{Ruhr-Universit\"at Bochum,
Fakult\"at f\"ur Mathematik,
Raum NA 2/33,
D-44780 Bochum}
\email{Olivier.Brunat@rub.de}

\author{Max Neunh\"offer}

\address{School of Mathematics and Statistics,
Mathematical Institute,
North Haugh,
St Andrews, Fife KY16 9SS,
Scotland, UK}

\email{neunhoef@mcs.st-and.ac.uk}


\date{3 October 2008}



\begin{abstract}
We denote by $A$ the ring of Laurent polynomials in the indeterminate
$v$ and by $K$ its field of fractions. In this paper, we are interested
in representation theory of the ``generic'' $q$-Schur algebra
$\Sa_q(n,r)$ over $A$. We will associate to every non-degenerate
symmetrising trace form $\tau$ on $K\Sa_q(n,r)$ a subalgebra $\J_{\tau}$
of $K\Sa_q(n,r)$ which is isomorphic to the ``asymptotic'' algebra
$\J(n,r)_A$ defined by J. Du. As a consequence, we give a new
criterion for James' conjecture.
\end{abstract}

\maketitle

\section{Introduction}

This article is concerned with the representation theory of the
``generic'' $q$-Schur algebra $\Sa_{q}(n,r)$ over $A=\Z[v,v^{-1}]$.
The $q$-Schur algebra was introduced by Dipper and James in~\cite{DJ}
and \cite{qtensor}. There is an interest in studying the representations
of this algebra, because they relate informations about the
modular representation theory of the finite general linear group
$\operatorname{GL}_n(q)$ and of the quantum groups.

Using a new basis of $\Sa_{q}(n,r)$ constructed in~\cite{DuKLBasis}
(which is analogous to the Kazhdan-Lusztig basis in Iwahori-Hecke
algebras), J. Du introduced in~\cite{DuAsymp} the asymptotic
algebra $\J(n,r)_A$ over $A$ and defined a homomorphism,
$\Phi:\Sa_q(n,r)\rightarrow\J(n,r)_A$, the so-called Du-Lusztig
homomorphism because its construction is similar to the Lusztig
homomorphism for Iwahori-Hecke algebras.

There is a relevant open question in the representation theory of the
$q$-Schur algebra, the so-called James' conjecture. A precise
formulation of this conjecture is recalled in Section~\ref{partjames}.
In~\cite{geckqschurjames} Meinolf Geck obtained a new formulation of
this conjecture. More precisely, for $k$ any field of characteristic
$\ell$ and for $R$ any integral domain with quotient field $k$,
if $q\in R$ is invertible, we can define the corresponding $q$-Schur algebra
$\Sa_q(n,r)_R$ over $R$ and its extension of
scalars $\Sa_q(n,r)_k$. Similarly, we can define $ \J(n,r)_k$. 

In \cite[1.2]{geckqschurjames} M.~Geck
has shown that James' conjecture holds if and only if, for $\ell > r$, 
the rank of the homomorphism
$\Phi_k:\Sa_q(n,r)_k\rightarrow\J(n,r)_k$ only depends on the
multiplicative order of $q$ in $k^{\times}$, but not on $\ell$.

Thus, in order to prove James' conjecture, it is relevant to understand
the rank of the Du-Lusztig homomorphism. The motivation of this paper
is to develop new methods allowing to study this rank. More precisely,
we will give a new construction of the asymptotic algebra. Indeed,
thanks to methods developed in~\cite{klwedder} by the second author
and adapted to our situation, we prove that $\J(n,r)_A$ is isomorphic
to an algebra $\J_{\tau}$, which only depends on the choice of a
non-degenerate symmetrising trace form $\tau$ on the semisimple algebra
$K\Sa_q(n,r)$ (here $K=\Q(v)$) such that 
\[ \Sa_q(n,r)\subseteq \J_{\tau}\subseteq K\Sa_q(n,r). \]
Our main tool is to use the structure of the left cell modules of
$\Sa_q(n,r)$ to construct an explicit Wedderburn basis of $K\Sa_q(n,r)$
(see Theorem~\ref{wedderburn}). The main result of this paper is
Theorem~\ref{newinter}.

The article is organized as follows. In Section~\ref{sec:qschur}, we
recall the definition of the ``generic'' $q$-Schur algebra and of
its analogue of the Kazhdan-Lusztig basis for Iwahori-Hecke algebras. In
Section~\ref{sec:conj} we prove that the $q$-Schur algebra satisfies
properties which are very similar to Lusztigs conjectures {\bf
P1},\ldots,\,{\bf P15} for Iwahori-Hecke algebras. In Section~\ref{irr} we
develop some tools to prove our main result in Section~\ref{sec:asym}.
Finally, in Section~\ref{partjames} we state a new criterion for
James' conjecture.


\section{The Iwahori-Hecke algebra of type A and the $q$-Schur algebra}
\label{sec:qschur}

Let $v$ be an indeterminate. We set $A=\Z[v,v^{-1}]$ to be the ring 
of Laurent polynomials in $v$ and $K := \Q(v)$ its field of fractions.
In order to introduce the $q$-Schur algebra over $A$, we have to
recall some definitions and properties about Iwahori-Hecke algebras. We 
follow~\cite{Uneq}.

\subsection{Iwahori-Hecke algebras and the Kazhdan-Lusztig basis}
\label{Hecke}

Let $(W,S)$ be a Coxeter group (here $S$
is the set of simple reflections). We define the corresponding Iwahori-Hecke
algebra $\Ha$ as the free $A$-module with basis $\{T_w\}_{w\in
W}$ satisfying \[\begin{array}{ll} T_wT_{w'}=T_{ww'}&\textrm{if}\
l(ww')=l(w)+l(w'),\\ (T_s-v)(T_s+v^{-1})=0&\textrm{for}\ s\in S,
\end{array}\] where $l$ is the length function on $W$. In~\cite[\S1]{KL79}
Kazhdan and Lusztig define an $A$-basis $\{C_w\ |\ w\in W\}$ of $\Ha$
which satisfies \[\overline{C}_w=C_w\quad\textrm{and}\quad C_w=\sum_{y\leq
w} p_{y,w} T_y\quad\textrm{for }w\in W,\] where $\leq$ is
the Bruhat-Chevalley order on $W$, and $^-:\Ha\rightarrow\Ha$ is the
involutive automorphism of $\Ha$ defined by $\overline{v}=v^{-1}$
and \smash{$\overline{\sum\limits_{w\in W}a_wT_w}=\sum\limits_{w\in
W}\overline{a}_wT_{w^{-1}}^{-1}$} and
$p_{y,w}\in \langle v^{k}\ |\ k\leq 0\rangle_{\Z}$ and $p_{w,w}=1$.

Note that we use the more modern notation from \cite{Uneq}, that is, our
elements $T_w$ here are the same as in \cite{Uneq} and were denoted by 
$v^{-l(w)} T_w$ in \cite{KL79}, and our elements $C_w$ here were
denoted by $C'_w$ in \cite{KL79} and by $c_w$ in \cite{Uneq}.

We denote by $g_{x,y,z}$ the structure constants of $\Ha$
with respect to the basis $\{C_w\ |\ w\in W\}$, that is, we have
\[C_xC_y=\sum_{z\in W}g_{x,y,z}C_z\quad\textrm{for }x,y\in W.\] We
define a relation $y\preccurlyeq_{\sL} w$ on $W$ by: either $y=w$ or
there is an $s\in S$ such that $g_{s,w,y}\neq 0$. Let $\leq_{\sL}$
be the transitive closure of the relation $\preccurlyeq_{\sL}$
and denote by $\sim_{\sL}$ the associated equivalence relation
on $W$. The classes for this relation are the so-called left cells.
Similarly, we define $\leq_{\sR}$ and $\sim_{\sR}$, and we call the 
corresponding equivalence classes right cells. For $y,w\in W$, we
write $y\leq_{\sLR}w$ if there is a sequence $y=y_0,\,y_1\ldots,y_n=w$
of elements of $W$ such that, for $i\in\{0,\ldots,n-1\}$, we
have $y_i\leq_{\sL}y_{i+1}$ or $y_i\leq_{\sR}y_{i+1}$. The classes
of the equivalence relation $\sim_{\sLR}$ on $W$ corresponding
to $\leq_{\sLR}$ are the so-called two-sided cells.

In \cite[\S3.6]{Uneq}, Lusztig shows that for $z\in W$, there is a unique 
integer $\a(z)$ such that for every $x,y\in W$, we have
$g_{x,y,z}v^{\a(z)}\in \Z[v^{-1}]$ and $g_{x,y,z}v^{\a(z)-1}\notin \Z[v^{-1}]$.
Moreover, for $z\in W$, we define $\Delta(z)=-\operatorname{deg}p_{1,z}$.
For $x,y,z\in W$, we write $\gamma_{x,y,z^{-1}}\in\Z$ for the coefficient
of $v^{\a(z)}$ in $g_{x,y,z}$ and we set \[\D=\{d\in W\ |\
\a(d)=\Delta(d)\},\]
the set of distinguished involutions.
In the case that $W$ is a finite Weyl group, an affine Weyl group, 
or a dihedral group, Lusztig proved that the following conjectures hold 
(see~\cite[\S\S15--17]{Uneq}):


\renewcommand{\arraystretch}{1}
\begin{tabular}{lp{4.3in}}
\textbf{P1}&For any $z\in W$ we have $\a(z)\leq\Delta(z)$.\\
\textbf{P2}&Let $x,\,y\in W$; if $\gamma_{x,y,d}\neq 0$ for some
$d\in\D$, then we have $x=y^{-1}$.\\
\textbf{P3}&If $y\in W$, there exists a unique $d\in\D$ such that
$\gamma_{y^{-1},y,d}\neq 0$.\\
\textbf{P4}&If $x\leq_{\sLR}y$, then $\a(x)\geq \a(y)$. \\
\textbf{P5}& If $d\in\D$ and $y\in W$ are such that
$\gamma_{y^{-1},y,d}\neq 0$, then $\gamma_{y^{-1},y,d}=\pm 1$.\\
\textbf{P6}& For $ d\in\D$, we have $d=d^{-1}$.\\
\textbf{P7}& For every $x,\,y,\,z\in W$, we have
$\gamma_{x,y,z}=\gamma_{y,z,x}=\gamma_{z,x,y}$.\\
\textbf{P8}& Let $x,\,y,\,z\in W$ be such that $\gamma_{x,y,z}\neq 0$,
then
$x \sim_{\sL}y^{-1}$, $y\sim_{\sL} z^{-1}$ and\\& $z \sim_{\sL} x^{-1}$.\\
\textbf{P9}&If $x\leq_{\sL}y$ and $\a(x)=\a(y)$, then $x\sim_{\sL}y$.\\
\textbf{P10}&If $x\leq_{\sR}y$ and $\a(x)=\a(y)$, then $x\sim_{\sR}y$.\\
\textbf{P11}&If $x\leq_{\sLR}y$ and $\a(x)=\a(y)$, then $x\sim_{\sLR}y$.\\
\textbf{P13}& Every left cell contains a unique element $d\in\D$ and
$\gamma_{y^{-1},y, d}\neq 0$ for\\& every $y\sim_{\sL}d$.\\
\textbf{P14}& For every $x\in W$, we have $x\sim_{\sLR}x^{-1}$.\\
\textbf{P15}& Let $v'$ be a second indeterminate and let
$g_{x,y,z}'\in\Z[v',v'^{-1}]$ be obtained from $g_{x,y,z}$ by the
substitution $v\mapsto v'$. If $x,x',y,w\in W$ satisfy $\a(w) =\a(y)$,
then 
$$\sum_{y'}g_{w,x',y'}'g_{x,y',y}=\sum_{y'}g_{x,w,y'}g_{y',x',y}'.$$
\end{tabular}

\noindent
Note that in this paper we only consider the case of type A, in which $W$ is
the symmetric group on $|S|+1$ points.
\subsection{The $q$-Schur algebra $\Sa_q(n,r)$}
\label{qschuralg}

In the following, we denote by $W$ the symmetric group
of degree $r$, and by $S$ the set of transpositions $s_i=(i,i+1)$ for $1\leq i\leq
r-1$ and $\Ha$ is the associated Iwahori-Hecke algebra as in~\S\ref{Hecke}.
Let $n,r\geq 1$, we denote by $\Lambda(n,r)$ the set of compositions of $r$ into
at most $n$ parts. For $\lambda\in\Lambda(n,r)$,  we denote
by $W_{\lambda}\subseteq W$ the corresponding Young subgroup.
For $\lambda,\mu\in\Lambda(n,r)$, we set $D_{\lambda,\mu}$ to be the set of
distinguished double coset representatives of $W$ with respect
to $W_{\lambda}$ and $W_{\mu}$. 
We set
\[M(n,r)=\{(\lambda,w,\mu)\ |\ \lambda,\mu\in\Lambda(n,r),\,w\in
D_{\lambda,\mu}\}.\]
For $\u a=(\lambda,w,\mu)\in M(n,r)$, we write $ro(\u a)=\lambda$
and $co(\u a)=\mu$ and we set $\u a^t=(\mu,w^{-1},\lambda)$.
For $\lambda,\,\mu\in\Lambda(n,r)$, we set $M_{\lambda,\mu}=\{\u
a\in M(n,r)\ |\ ro(\u a)=\lambda,\,co(\u a)=\mu\}$. We
remark that if $w\in D_{\lambda,\mu}$, then the double
coset $W_{\lambda}w W_{\mu}$ has a unique longest element. To
prove this, we can proceed as follows: we denote by $w_0$ the longest
element of $W$, then $^{w_0}W_{\mu}=W_{\widetilde{\mu}}$.
Here $\widetilde{\mu}=(\mu_s,\mu_{s-1},\ldots,\mu_1)$,
where $\mu=(\mu_1,\ldots,\mu_s)$.
Moreover, $r_{w_0}:W\rightarrow W,\,x\mapsto xw_0$ induces a
bijection from the double coset $W_{\lambda}ww_0 W_{\widetilde{\mu}}$
to the double coset $W_{\lambda}w W_{\mu}$. Thanks to~\cite[11.3]{Uneq},
we deduce that $r_{w_0}$ reverses the Bruhat-order. Since the
double coset $W_{\lambda}ww_0 W_{\widetilde{\mu}}$ has a unique
element of minimal length, the result follows. 
We write $D_{\lambda,\mu}^+$
for the set of double coset representatives of maximal length.
We denote by $\ell_{\lambda,\mu}$ the bijection from $D_{\lambda,\mu}$
to $D_{\lambda,\mu}^+$ that associates to the representative of
minimal length $w$ of the double coset $W_{\lambda}w W_{\mu}$ the
representative of maximal length. We remark that if $w\in
D_{\lambda,\mu}$, then $w^{-1}\in D_{\mu,\lambda}$. Moreover, we have
\[ \ell_{\lambda,\mu}(w)^{-1}=\ell_{\mu,\lambda}(w^{-1}). \] 
In the following, we set $\sigma(\u a):=\ell_{\lambda,\mu}(w)$ for $\u a
= (\lambda,w,\mu)$.

We now recall the definition of the $q$-Schur algebra $\Sa_q(n,r)$
introduced by Dipper and James in~\cite{DJ}. We set $q=v^2$, then 
the $q$-Schur algebra $\Sa_q(n,r)$
of degree $(n,r)$ is the endomorphism algebra
\[\Sa_q(n,r)=\operatorname{End}_{\Ha}\left(\bigoplus_{\lambda\in\Lambda(
n,r)}x_{\lambda}\Ha\right),\] 
where $x_{\lambda}=\sum\limits_{w\in W_{\lambda}}v^{l(w)}T_w\in\Ha$. 
In~\cite[3.4]{DJRepHecke} Dipper and James prove 
that $\Sa_q(n,r)$ has a standard
basis $\{\phi_{\lambda,\mu}^w\ |\ (\lambda,w,\mu)\in M(n,r)\}$ indexed
by the set $M(n,r)$, which plays the same role as the
basis $\{T_w\ |\ w\in W\}$ for the Iwahori-Hecke algebra $\Ha$. Moreover,
in~\cite{DuKLBasis} Du proves that $\Sa_q(n,r)$ has another basis 
$\{\thu a\mid \u a\in M(n,r)\}$ whose construction is  analogous to the
Kazhdan-Lusztig basis of $\Ha$. We denote by $f_{\u a,\u b,\u c}\in A$
the structure constants with respect to this basis, that is, we have
\[ \thu a\thu b = \sum_{\u c\in M(n,r)} f_{\u a,\u b,\u c}\thu c \qquad
   \mbox{for all } \u a, \u b \in M(n,r).\]

We recall the following lemma:

\begin{Lemma}\label{scqschur}
We have $f_{\u a,\u b,\u c}\neq 0$ only if $co(\u a)=ro(\u b)$ 
and $(ro(\u a),co(\u b))=(ro(\u c),co(\u c))$. In this case, we 
have
\[f_{\u a,\u b,\u c}=h_{\mu}^{-1}g_{\sigma(\u a),\sigma(\u b),\sigma(\u c)}.\]
where $\mu = co(\u a) = ro(\u b)$ and 
$h_{\mu}=\sum\limits_{w\in W_{\mu}}v^{2l(w)-l(w_{\mu})}$ (here $w_{\mu}$
denotes the longest element in $W$) 
and $g_{\sigma(\u a),\sigma(\u b),\sigma(\u c)}$ is the structure
constant of\/ $\Ha$ defined in Section~\ref{Hecke}.
\end{Lemma}

\begin{proof}
See \cite[Prop.~3.4]{DuKLBasis}. We want to explain why we have a further
hypothesis here than in \cite[Prop.~3.4]{DuKLBasis}: For $\u a =
(\lambda,w,\mu) \in M(n,r)$ the element $\thu a$ is by definition a linear
combination of basis elements $\phi_{\lambda,\mu}^z$ for $z \in
\D_{\lambda,\mu}$. Thus, viewed as endomorphism of $\bigoplus_{\lambda
\in \Lambda(n,r)} x_\lambda \Ha$ it vanishes on all summands except
on $x_\mu \Ha$ and maps into the summand $x_\lambda \Ha$. Thus, if either
$co(\u a) \neq ro(\u b)$ or $(ro(\u a),co(\u b)) \neq (ro(\u c),co(\u c))$,
the structure constant $f_{\u a,\u b,\u c}$ vanishes also. If both
equations hold, the proof in \cite[Prop.~3.4]{DuKLBasis} works
using $g_{\sigma(\u a),\sigma(\u b),\sigma(\u c)}$.

We are not claiming that \cite[Prop.~3.4]{DuKLBasis} is wrong as stated
there. However, the notation $g_{\u a,\u b,\u c}$ there needs proper
interpretation (see \cite[Section 3.3]{DuKLBasis}), a problem we 
avoid here.
\end{proof}

\begin{remark}
To further explain the just mentioned change of notation, consider the
following: Let $n = r = 3$, $\lambda := (2,1,0)$, $\mu := (1,1,1)$, and
$\nu := (2,1,0)$. Then $W$ is the symmetric group on $3$ letters, generated
by the two Coxeter generators $s_1 = (1,2)$ and $s_2 = (2,3)$. 
Thus $\D_{\lambda,\mu}^+ := \{ s_1, s_1s_2, s_1s_2s_1 \}$, $\D_{\mu,\nu}^+ = 
\{ s_1, s_2s_1, s_1 s_2 s_1 \}$ and $\D_{\lambda,\nu}^+ = \{ s_1, s_1s_2s_1 \}$.

By the relations, we have $T_{s_1} \cdot T_{s_2s_1} = T_{s_1s_2s_1}$ and
thus $g_{s_1,s_2s_1,s_1s_2s_1} = 1$. We now set $\u a := (\lambda,\id,\mu)$,
$\u b := (\mu,s_2,\nu)$ and $\u c := (\lambda,s_2,\nu)$. Thus, we get
\[ f_{\u a,\u b,\u c} = 1 \cdot g_{\sigma(\u a),\sigma(\u b),\sigma(\u c)}
= g_{s_1,s_2s_1,s_1s_2s_1} = 1, \] 
since $h_\mu = 1$ here.

However, if we set $\u{a'} := (\mu,s_1,\mu)$, then 
$f_{\u{a'},\u b,\u c} = 0$, because of $ro(\u{a'}) \neq ro(\u c)$
and the arguments in the proof of Lemma~\ref{scqschur}. On the other
hand, we have $ro(\u{a'}) = co(\u b)$ and 
$g_{\sigma(\u{a'}),\sigma(\u b),\sigma(\u c)} =
 g_{s_1,s_2s_1,s_1s_2s_1} = 1$. This shows, that we indeed need all
the hypothesis in Lemma~\ref{scqschur}. The statement in 
\cite[Prop.~3.4]{DuKLBasis} is true if one interprets $g_{\u{a'},\u b,\u c}$
to be zero.
\end{remark}

\begin{Def}[The $\a$-function and the distinguished elements]
\label{defD}
Following~\cite[Section 2]{DuAsymp}, we extend the $\a$-function to $M(n,r)$ by
setting $\a(\u a)=\a(\sigma(\u a))$ for every $\u a\in M(n,r)$ and we
extend the set $\D$ to the set \[\D(n,r)=\{\u d\in
M(n,r)\ |\ co(\u d)=ro(\u d),\, \sigma(\u d)\in\D\}.\]
Moreover,  for every $\u a,\,\u b,\,\u c\in M(n,r)$, we define
\[\gamma_{\u a,\u b,\u c^t}=\left\{
\begin{array}{ll}
\gamma_{\sigma(\u a),\sigma(\u b),\sigma(\u c^t)}
=\gamma_{\sigma(\u a),\sigma(\u b),\sigma(\u c)^{-1}} &\textrm{if } f_{\u a,\u
b,\u c}\neq 0,\\
0&\textrm{otherwise}.
\end{array}\right.\]
\end{Def}

\begin{remark}
Note that our definition for $\gamma_{\u a,\u b,\u c}$ differs slightly
from the one in \cite[Section 2.2]{DuAsymp}. His $\gamma_{\u a, \u b, \u
c}$ is our $\gamma_{\u a, \u b, \u c^t}$. With our definition we follow
the setup in \cite{Uneq} more closely and get nicer cyclic symmetries
in our formulas.
\end{remark}

\begin{remark}
In comparison to \cite[Section 2.1]{DuAsymp} we added the explicit hypothesis
for the elements $\u d \in \D(n,r)$ that $ro(\u d) = co(\u d)$. However,
this hypothesis is implicit in \cite{DuAsymp}, since otherwise the
statements in \cite[4.1,(a)--(d)]{DuAsymp} and some others would not be true.
\end{remark}

Now, for $\u a,\,\u b\in M(n,r)$, if there is $\u c\in M(n,r)$ such that
$f_{\u c,\u b,\u a}\neq 0$ then we write $\u a\leq_{\sL}\u b$ . We
define $\leq_{\sR}$ by
$\u a\leq_{\sR}\u b$ if and only if $\u a^t\leq_{\sL}\u b^t$. Moreover,
we define $\leq_{\sLR}$ as in the Iwahori-Hecke
algebra case. These relations induce corresponding equivalence relations 
$\sim_{\sL}$, $\sim_{\sR}$ and $\sim_{\sLR}$. We
call the corresponding equivalence classes the left, right and two-sided
cells of $M(n,r)$ respectively.

Let $\Gamma$ be a left cell of $M(n,r)$. We set 
\[ \mathcal{S}_{\leq\Gamma}=\sum_{\u b\leq_{\sL}\u a}A\theta_{\u
b}\quad\textrm{and}\quad\mathcal{S}_{<\Gamma}=\sum_{\u b\leq_{\sL}\u
a,\ \u b\not\sim_{\sL}\u a}A\theta_{\u b},\] 
for some~$\u a\in\Gamma$, both are clearly left ideals of $\Sa_q(n,r)$
by the definition of $\le_{\sL}$.
Then the left cell module $\LC^{(\Gamma)}$ corresponding to $\Gamma$ is
defined as the quotient $\mathcal{S}_{\leq\Gamma}/\mathcal{S}_{<\Gamma}$. 

We define the right cell module $\RC^{(\Gamma)}$
corresponding to a right cell $\Gamma$ of $M(n,r)$ similarly.
To see that we get right ideals we have to use Lemma~\ref{scqschur}
and $g_{x,y,z} = g_{y^{-1},x^{-1},z^{-1}}$ for $x,y,z \in W$ (see 
\cite[13.2.(e)]{Uneq}) together with $\sigma(\u a^t) = \sigma(\u a)^{-1}$.
This implies $f_{\u a,\u b,\u c} = 0$ if and only if 
$f_{\u b^t,\u a^t,\u c^t} = 0$.

\section{Lusztig's conjectures for the $q$-Schur
algebra}\label{sec:conj}

In this section, we prove that the $q$-Schur algebra satisfies properties
very similar to $\textbf{P1},\ldots,\textbf{P15}$ for the
Iwahori-Hecke algebra. First, we give some preliminary results.

\begin{Lemma}\label{sigmale}
If $\u a\leq_{\sL}\u b$ (resp. $\leq_{\sR}$, $\leq_{\sLR}$),
then $\sigma(\u a)\leq_{\sL}\sigma(\u b)$ (resp. $\leq_{\sR}$,
$\leq_{\sLR}$).
\end{Lemma}
\begin{proof}
Since $\u a\leq_{\sL}\u b$, there is $\u c\in M(n,r)$
such that $f_{\u c,\u b,\u a}\neq 0$.
But we have $f_{\u c,\u b,\u a}=
h_{co(a)}^{-1}g_{\sigma(\u c)
,\sigma(\u b),\sigma(\u a)}$ with $h_{co(a)}^{-1}\neq 0$.
Thus $g_{\sigma(\u c),\sigma(\u b),\sigma(\u a)}\neq 0$
and $\sigma(\u a)\leq_{\sL}\sigma(\u b)$.
\end{proof}

\begin{Lemma}\label{roco}
If $\u a\leq_{\sL}\u b$, then $co(\u a)=co(\u b)$. If $\u a \leq_{\sR} \u
b$, then $ro(\u a) = ro(\u b)$.
\end{Lemma}

\begin{proof}Since $\u a\leq_{\sL}\u b$ there is $\u c\in M(n,r)$ such
that $f_{\u c,\u b, \u a}\neq 0$. From Lemma~\ref{scqschur} follows that
$(ro(\u a),co(\u a))=(ro(\u c), co(\u b))$ and the result is proved.
\end{proof}

\begin{Lemma}\label{lem3}
Let $\lambda,\,\mu,\,\nu\in\Lambda(n,r)$, $x\in D_{\lambda,\mu}^+$ and
$y\in D_{\mu,\nu}^+$. If $g_{x,y,z}\neq 0$ for some $z\in W$, then
$z\in D_{\lambda,\nu}^+$.
\end{Lemma}

\begin{proof}
For $\lambda \in \Lambda(n,r)$ we set $S_{\lambda} := W_{\lambda} \cap S$, the
set of Coxeter generators of the parabolic subgroup $W_\lambda$.
Let $x \in \D_{\lambda,\mu}^+$ and $y \in \D_{\mu,\nu}^+$ and 
$g_{x,y,z} \neq 0$. On one hand, this means that $l(sx) < l(x)$ for all 
$s \in S_\lambda$ and $l(ys) < l(y)$ for all $s \in S_\nu$. On the other
hand, we get $z \le_{\sL} y$ and $z \le_{\sR} x$ and thus
$l(zs) < l(z)$ for all $s \in S$ with $l(ys)<l(y)$ and $l(sz) < l(s)$ for all
$s \in S$ with $l(sx) < l(x)$ by \cite[Lemma 8.6]{Uneq}. Thus we have
in particular that $l(zs) < l(z)$ for all $s \in S_\nu$ and $l(sz) < l(z)$
for all $s \in S_\lambda$. Hence $z$ is
the longest element in its $W_\lambda$-$W_\nu$-double coset in $W$.
\end{proof}

\begin{Lemma}\label{rightcells}
We have $\u a\leq_{\sR}\u b$ if and only if there is a $\u c \in M(n,r)$
with $f_{\u b,\u c,\u a} \neq 0$.
\end{Lemma}

\begin{proof}
By definition, $\u a\leq_{\sR} \u b$ is equivalent to $\u a^t \leq_{\sL} \u
b^t$.
This in turn means that there is a $\u c \in M(n,r)$ such that
$f_{\u c^t, \u b^t, \u a^t} \neq 0$.
As mentioned at the end of Section~\ref{qschuralg} we have 
$f_{\u b,\u c, \u a} = 0$ if and only if $f_{\u c^t,\u b^t,\u a^t} = 0$
which directly implies the statement in the lemma.
\end{proof}

\begin{Prop} The following properties hold for the $q$-Schur algebra:

\renewcommand{\arraystretch}{1}
\begin{tabular}{lp{4.3in}}
\textbf{Q1}&For any $\u a\in M(n,r)$ we have $\a(\u
a)\leq\Delta(\sigma(\u a))$.\\
\textbf{Q2}&If $\gamma_{\u a,\u b, \u d}\neq 0$ for some $\u
d\in\D(n,r)$, then we have $\u b=\u a^t$.\\
\textbf{Q3}&For every $\u a\in M(n,r)$, there is a unique $\u
d\in\D(n,r)$ with $\gamma_{\u a^t,\u a,\u d}\neq 0$.\\
\textbf{Q4}&If $\u a\leq_{\sLR}\u b$, then $\a(\u a)\geq \a(\u b)$. \\
\textbf{Q5}& If $\u d\in\D(n,r)$ and $\u a\in M(n,r)$ are such
that $\gamma_{\u a^t,\u a,\u d}\neq 0$, then $\gamma_{\u a^t,\u a,\u
d}=1$.\\
\textbf{Q6}& For $\u d\in\D(n,r)$, we have $\u d=\u d^{t}$.\\
\textbf{Q7}& For every $\u a,\,\u b,\,\u c\in M(n,r)$, we
have $\gamma_{\u a,\u b,\u c}=\gamma_{\u b,\u c,\u a}=\gamma_{\u
c,\u a,\u b}$.\\
\textbf{Q8}& Let $\u a,\,\u b,\,\u c\in M(n,r)$ be such that $\gamma_{\u a,
\u b,\u c}\neq 0$, then $\u a \sim_{\sL}\u b^t$, $\u b\sim_{\sL} \u
c^t$\\& 
and $\u c \sim_{\sL}\u a^t$.\\
\textbf{Q9}&If $\u a\leq_{\sL}\u b$ and $\a(\u a)=\a(\u b)$, then $\u a\sim_{\sL}\u b$.\\
\textbf{Q10}&If $\u a\leq_{\sR}\u b$ and $\a(\u a)=\a(\u b)$, then $\u a\sim_{\sR}\u b$.\\
\textbf{Q11}&If $\u a\leq_{\sLR}\u b$ and $\a(\u a)=\a(\u b)$, then $\u a\sim_{\sLR}\u b$.\\
\textbf{Q13}& Every left cell contains a unique element $\u d\in\D(n,r)$ and 
$\gamma_{\u a^t,\u a,\u d}\neq 0$\\& for every $\u a\sim_{\sL}\u d$.\\
\textbf{Q14}& For every $\u a\in M(n,r)$, we have $\u a\sim_{\sLR}\u
a^{t}$.\\
\textbf{Q15}& Let $v'$ be a second indeterminate and let
$f_{x,y,z}'\in\Z[v',v'^{-1}]$ be obtained from $f_{x,y,z}$ by the
substitution $v\mapsto v'$. If $\u a,\u a',\u b,\u c\in W$ satisfy
$\a(\u c) =\a(\u b)$,
then 
$$\sum_{\u b'}f_{\u c,\u a',\u b'}'f_{\u a,\u b',\u b}=\sum_{\u b'}f_{\u
a,\u c,\u b'}f_{\u b',\u a',\u b}'.$$
\end{tabular}
\end{Prop}

\begin{proof}
We note that {\bf Q1} is a direct consequence of Property \textbf{P1}.

We now will prove Property {\bf Q2}. We suppose that $\gamma_{\u
a,\u b,\u d}\neq 0$ for some $\u a,\,\u b\in M(n,r)$ and $\u
d\in\D(n,r)$. Since $\gamma_{\u a,\u b,\u d}\neq 0$, it follows
that $f_{\u a,\u b,\u d}\neq 0$. Thus we have $co(\u a)=ro(\u b)$,
$ro(\u a)=ro(\u d)$ and $co(\u b)=co(\u d)$ by Lemma~\ref{scqschur}.
But $co(\u d)=ro(\u d)$ implies $ro(\u a)=co(\u b)$. We now
write $\u a=(\lambda,w_a,\mu)$ and $\u b=(\mu,w_b,\lambda)$. We
have $\gamma_{\u a,\u b,\u d}=\gamma_{\sigma(\u a),\sigma(\u
b),\sigma(\u d)}$. From $\sigma(\u d)\in\D$ we deduce using
\textbf{P2} that $\sigma(\u a)=\sigma(\u b)^{-1}$. It follows that
$\ell_{\lambda,\mu}(w_a)=\ell_{\mu,\lambda}(w_b)^{-1}=\ell_{\lambda,\mu}
(w_b^{-1})$, we get $w_a=w_b^{-1}$ and thus \textbf{Q2} holds.

Let $\u a=(\lambda,w,\mu)\in M(n,r)$. Thanks to Property~\textbf{P3}, there is a
unique $d\in\D$ such that $\gamma_{\sigma(\u a)^{-1},\sigma(\u
a),d}\neq 0$. Since $\sigma(\u a)^{-1}=\sigma(\u a^t)$, we deduce
that $g_{\sigma(\u a^t),\sigma(\u a),d}\neq 0$.
But $\sigma(\u
a^t)\in D_{\mu,\lambda}^+$ and $\sigma(\u a)\in D_{\lambda,\mu}^+$,
then Lemma~\ref{lem3} gives $d\in D_{\mu,\mu}^+$. We
denote by $\widetilde{d}$ the representative of minimal length
of the coset $W_{\mu}dW_{\mu}$ and we set $\u
d := (\mu,\widetilde{d},\mu)$. Then $\u d\in\D(n,r)$
and $\sigma(\u d)=d$. It follows that $\gamma_{\u a^t,\u a,\u d}\neq 0$
and thus \textbf{Q3} holds.

The property~\textbf{Q4} follows from~\textbf{P4} and Lemma~\ref{sigmale}.
The property~\textbf{Q5} directly follows from~\textbf{P5}, since
in our case $W$ is of type A and thus all coefficients of all 
Kazhdan-Lusztig polynomials are
non-negative by \cite[15.1]{Uneq}.

Let $\u d=(\lambda,w,\lambda)\in\D(n,r)$; we have $\sigma(\u
d)\in\D$, thus \textbf{P6} gives $\sigma(\u d)^{-1}=\sigma(\u
d)$. Therefore, we have $\ell_{\lambda,\lambda}(w)=\sigma(\u
d)^{-1}=\sigma(\u d^t)=\ell_{\lambda,\lambda}(w^{-1})$,
and it follows that $w=w^{-1}$; thus \textbf{Q6} holds. The
property~\textbf{Q7} follows directly from~\textbf{P7}.

Suppose that $\gamma_{\u a,\u b,\u c}\neq 0$ for some 
$\u a,\,\u b,\,\u c\in M(n,r)$, then $f_{\u a,\u b,\u c^t}\neq 0$ and
it follows that $co(\u a)=ro(\u b)$ and $(ro(\u a),co(\u b))=(ro(\u
c^t),co(\u c^t))$. Then we have 
\begin{eqnarray*}
f_{\u b^t,\u a^t,\u c}&=&h_{co(\u a)}g_{\sigma(\u b^t),\sigma(\u a^t),
    \sigma(\u c)}\\
  &=&h_{co(\u a)}g_{\sigma(\u b)^{-1},\sigma(\u a)^{-1},\sigma(\u c)}\\
  &=&h_{co(\u a)}g_{\sigma(\u a),\sigma(\u b),\sigma(\u c)^{-1}}\\
  &=&f_{\u a,\u b,\u c^t}. 
\end{eqnarray*}
It follows that $\u c^t\leq_{\sL} \u b$ and $\u c\leq_{\sL}\u a^t$.
Using~\textbf{Q7} and the same arguments applied to $\gamma_{\u b,\u c,\u
a} = \gamma_{\u c,\u a,\u b} \neq 0$, we deduce that $\u a \sim_{\sL}\u b^t$,
$\u b\sim_{\sL} \u c^t$ and $\u c \sim_{\sL} \u a^t$.
Thus~\textbf{Q8} holds.

Next we prove \textbf{Q13}. Let $\u a \in M(n,r)$.
By \textbf{Q3}
there is a unique $\u d \in \D(n,r)$ with $\gamma_{\u a^t,\u a,\u d}
\neq 0$ and for this $\u d$ holds $\u a\sim_{\sL}\u d$ by \textbf{Q8}.
But for $\u d, \u d' \in \D(n,r)$ with $\u d \sim_{\sL} \u d'$ we conclude
$ro(\u d) = co(\u d) = co(\u d') = ro(\u d')$ using Lemma~\ref{roco}
and $\sigma(\u d) = \sigma(\u d')$ using \textbf{P13} since
$\sigma(\u d) \sim_{\sL} \sigma(\u d')$ because of Lemma~\ref{sigmale}. 
Thus we have proved \textbf{Q13}.

Now we prove \textbf{Q9}.
Let $\u a, \u b\in M(n,r)$ with
$\u a \leq_{\sL} \u b$ and $\a(\u a) = \a(\u b)$. 
We denote the unique element of $\D(n,r)$ in the left cell of $\u a$
by $\u d_a$ (resp.~$\u d_b$ for $\u b$).
Using~\textbf{Q4} we deduce that $\a(\u d_a)=\a(\u
a)$ and $\a(\u d_b)=\a(\u b)$. Moreover, we have $\u d_a\leq_{\sL} \u d_b$.
Thus using Lemma~\ref{sigmale} shows that $\sigma(\u d_a)\leq_{\sL}
\sigma(\u d_b)$. Hence, using Property \textbf{P9}, we have
$\sigma(\u d_a)\sim_L\sigma(\u d_b)$.
However, $\sigma(\u d_a)$ and $\sigma(\u d_b)$ lie in $\D$.
Therefore, using \textbf{P13} in the Iwahori-Hecke algebra, we deduce
that $\sigma(\u d_a)=\sigma(\u d_b)$. 
We now prove that $f_{\u d_a,\u d_a,\u d_b}\neq 0$. Since 
$ro(\u d_a)=co(\u d_a)=co(\u d_b)=ro(\u d_b)$ (thanks
to Lemma~\ref{roco}), 
we deduce that
\[f_{\u d_a,\u d_a,\u d_b}=h_{co(\u d_a)}^{-1}g_{\sigma(\u
d_a),\sigma(\u d_a),\sigma(
\u d_b)}.\] Using~\textbf{P13}, we deduce that $\gamma_{\sigma(\u
d_a)^{-1},\sigma(\u d_a),\sigma(\u d_b)}\neq 0$;
hence $g_{\sigma(\u d_a),\sigma(\u d_a),\sigma(\u d_b)}\neq 0$.
Since $h_{co(\u d_a)}^{-1}\neq 0$, it follows that $f_{\u d_a,\u d_a,\u
d_b}\neq 0$.
Hence $\u d_b\leq_{\sL} \u d_a$ and~\textbf{Q9} follows.

Property \textbf{Q10} follows from \textbf{Q9} by transposition since
$\a(\u a) = \a(\u a^t)$ for all $\u a \in M(n,r)$ (use \cite[13.9
(a)]{Uneq}).
Property \textbf{Q11} follows from \textbf{Q9} and \textbf{Q10} and induction.

Let $\u a\in M(n,r)$ and $\u d\in \D(n,r)$ be the unique element such 
that $\u a\sim_{\sL} \u d$ 
given by~\textbf{Q13}. Then $\u a^t\sim_{\sR} \u d^t=\u d$
and~\textbf{Q14} holds.

Finally, we prove \textbf{Q15}. We first remark that $f_{\u c,\u a',\u
b'}'\neq 0$ if and only if $f_{\u
a,\u c,\u b'}\neq 0$, and $f_{\u a,\u b',\u b}\neq 0$ if and only if
$f_{\u b',\u a',\u b}'\neq 0$. Moreover if $f_{\u c,\u a',\u
b'}'\neq 0$, then 
 $f'_{\u c, \u
a',\u b'}=h_{ro(\u a')}'g_{\sigma(\u c), \sigma(\u
a'),\sigma(\u b')}$ and $f_{\u
a,\u c,\u b'}=h_{co(\u a)}g_{\sigma(\u a), \sigma(\u
c),\sigma(\u b')}$. If $f_{\u a,\u c,\u
b'}\neq 0$, then 
 $f_{\u a, \u
c,\u b'}=h_{co(\u a)}g_{\sigma(\u a), \sigma(\u
c),\sigma(\u b')}$ and $f_{\u
b',\u a',\u b}'=h_{ro(\u a')}'g_{\sigma(\u b'), \sigma(\u
a'),\sigma(\u b)}$. Here $h'_{\mu}$ is obtained from $h_{\mu}$ by the
substitution $v\mapsto v'$. We note that $h_{ro(\u a')}$ and $h_{co(\u
a)}$ do not depend on $\u b'$. 
It follows from \textbf{P15} that
$$\begin{array}{lcl}
\sum\limits_{\u b'}f_{\u c,\u a',\u b'}'f_{\u a,\u b',\u b}&=&
h_{ro(\u a')}h_{co(\u a)}
\sum\limits_{\u b'}g_{\sigma(\u c),\sigma(\u a'),\sigma(\u b')}'
g_{\sigma(\u a),\sigma(\u b'),\sigma(\u b)}\\
&=&
h_{ro(\u a')}h_{co(\u a)}
\sum\limits_{\u b'}g_{\sigma(\u
a),\sigma(\u c),\sigma(\u b')}f_{\sigma(\u b'),\sigma(\u a'),\sigma(\u
b)}'\\
&=&
\sum\limits_{\u b'}f_{\u
a,\u c,\u b'}f_{\u b',\u a',\u b}'.
\end{array}$$

\end{proof}

\begin{Prop}
\label{LReq}
If $\u a\sim_{\sL}\u b$ and $\u a\sim_{\sR}\u b$, then $\u a=\u b$.
\end{Prop}
\begin{proof}
Let $\u a=(\lambda_a,w_a,\mu_a)$ and $\u b=(\lambda_b,w_b,\mu_b)$
be such that $\u a\sim_{\sL}\u b$ and $\u a\sim_{\sR}\u b$. We
have $\u a\leq_{\sL}\u b$ and $\u a^t\leq_{\sL}\u b^t$, then
using Lemma~\ref{roco} we deduce that $\mu_a=\mu_b$
and $\lambda_a=\lambda_b$. Using Lemma~\ref{sigmale}, we deduce that $\sigma(\u
a)\sim_{\sL}\sigma(\u b)$ and $\sigma(\u a)\sim_{\sR}\sigma(\u b)$.
Since $\Ha$ is of type $A$, it follows that $\sigma(\u a)=\sigma(\u b)$,
that is
$\ell_{\lambda_a,\mu_a}(w_a)=\ell_{\lambda_a,\mu_a}(w_b)
=\ell_{\lambda_b,\mu_b}(w_b)$. Hence we get
$w_a=w_b$.
\end{proof}

\section{Irreducible cell modules and dual basis}\label{irr}

In this section we view the extension of scalars $K \Sa_q(n,r)$ of the
$q$-Schur algebra $\Sa_q(n,r)$ as a symmetric algebra. This is possible, 
since it is semisimple (see \cite[(9.8)]{CR1}). 
We can take as symmetrising trace form any
$K$-linear form $\tau : K\Sa_q(n,r) \to K$ that is a $K$-linear
combination 
\[ \tau = \sum_{\chi \in \Irr(K \Sa_q(n,r))} \frac{\chi}{c_\chi} \]
of the irreducible characters where the $c_\chi$ are non-zero constants,
the so-called Schur elements (see \cite[7.1.1 and 7.2.6]{GP}). 
Clearly, $\tau$ is non-degenerate.

Having fixed $\tau$, we denote for any $K$-basis $(B_{\u a})_{\u a \in
M(n,r)}$ of $K\Sa_q(n,r)$ its dual basis with respect to $\tau$ by
$(B^\vee_{\u b})_{\u b \in M(n,r)}$. That is, we have $\tau(B_{\u a}
\cdot B^\vee_{\u b}) = \tau(B^\vee_{\u b} \cdot B_{\u a}) = \delta_{\u
a, \u b}$ for all $\u a, \u b \in M(n,r)$. Note that this immediately
implies that we can write every element $x \in K \Sa_q(n,r)$ in the
following form:
\begin{equation}
\label{theformula}
x = \sum_{\u a \in M(n,r)} \tau(x \cdot B^\vee_{\u a}) B_{\u a}
     = \sum_{\u a \in M(n,r)} \tau(x \cdot B_{\u a}) B^\vee_{\u a} 
\end{equation}
(just write $x$ as a linear combination of the $B_{\u a}$, multiply by
some $B_{\u b}$ and apply $\tau$).
\begin{remark}
\label{apresformule}
We have $f_{\u a, \u b, \u c} = \tau( \thu a \cdot \thu b \cdot
\thu{c}^\vee )$ for all $\u a, \u b, \u c \in M(n,r)$. Moreover, we
note that Formula~(\ref{theformula}) immediately gives us nice formulas
for the matrix representations coming from the left cell modules. For a
left cell $\Gamma$ and an element $h \in \Sa_q(n,r)$ the representing
matrix of $h$ on the left cell module $\LC^{(\Gamma)}$ with respect to
the basis $\{ \theta_{\u a} + \Sa_{<\Gamma} \mid \u a \in \Gamma \}$ is
$ \left(\tau( \theta_{\u b}^\vee \cdot h \cdot \theta_{\u a} ) \right)
_{\u b, \u a \in \Gamma}$ since $h \cdot \theta_{\u a} = \sum_{\u b \in
M(n,r)} \tau( \theta_{\u b}^\vee \cdot h \cdot \theta_{\u a} ) \cdot
\theta_{\u b}$ and it is enough to sum over those $\u b$ with $\u b
\leq_{\sL} \u a$.
\end{remark}

\begin{Lemma}[Characterisation of $\le_{\sL}$ and $\le_{\sR}$]
\label{charLR}
We have $\u a \le_{\sL} \u b$ if and only if $\thu b \thu a^\vee
\neq 0$ and $\u a \le_{\sR} \u b$ if and only if 
$\thu a^\vee \thu b \neq 0$.
\end{Lemma}

\begin{proof}
We only show the version with $\le_{\sL}$, the other is completely analogous 
thanks to Lemma~\ref{rightcells}.
If $\u a \le_{\sL} \u b$ there exists
a $\u c \in M(n,r)$ with $f_{\u c, \u b, \u a} = 
\tau( \thu c \thu b \thu a^\vee) \neq 0$ which implies 
$\thu b \thu a^\vee \neq 0$. If we assume the latter,
then by the non-degeneracy of $\tau$ there is some $\u c \in M(n,r)$
with $\tau(\thu c \thu b \thu a^\vee) \neq 0$ and
$\u a \le_{\sL} \u b$ follows.
\end{proof}

The other major ingredient is the fact that cell modules are simple, more
precisely:

\begin{Theo}[{Simple cell modules, see \cite{DuCanGL} or
\label{SimpleCell}
\cite[4.3]{DuAsymp}}]
Let $\Gamma$ be a left cell and recall $K = \Q(v)$. The extension of scalars
$K\LC^{(\Gamma)}$ of the left cell module
$\LC^{(\Gamma)}$ for a left cell\/ $\Gamma$ is a simple $K\Sa_q(n,r)$-module.
\end{Theo}

\begin{proof} See \cite{DuCanGL} or \cite[4.3]{DuAsymp}.
\end{proof}

\begin{remark}\label{coefinA}
This in particular implies that all simple $K\Sa_q(n,r)$-modules can be
realised over the ring $A$, since their corresponding representating
matrices involve only structure constants of $\Sa_q(n,r)$.
\end{remark}

We now directly obtain useful algebra elements by using the simple
cell modules:

\begin{Theo}[Basis of an isotypic component]
\label{isotypbasis}
Let $\Gamma$ be a left cell and $\chi$ the corresponding irreducible
character of the left cell module $\LC^{(\Gamma)}$, then the elements
\[ \left( c_\chi^{-1} \bas ab \right)_{\u a, \u b \in \Gamma} \]
are $K$-linearly independent and span the isotypic component of $K\Sa_q(n,r)$
belonging to the character $\chi$. Furthermore, we have the relations
\[ \left( c_\chi^{-1} \bas ab \right) \cdot
   \left( c_\chi^{-1} \bas{a'}{b'} \right)
 = \delta_{\u b, \u a'} \cdot 
   c_\chi^{-1} \bas a{b'} \]
for all $\u a, \u b, \u a', \u b' \in \Gamma$. That is, these elements
form a matrix unit for the isotypic component of $K\Sa_q(n,r)$ corresponding
to the simple module $K\LC^{(\Gamma)}$.
\end{Theo}

\begin{proof}
By \cite[7.2.7]{GP} we get a matrix unit for the isotypic component
of $K\Sa_q(n,r)$ corresponding to the simple module $K \LC^{(\Gamma)}$
by the elements
\[ \frac{1}{c_\chi} \sum_{\u c \in M(n,r)}
   \tau( \thu b^\vee \cdot \thu c \cdot \thu a )
      \cdot \thu c^\vee
 = \frac{1}{c_\chi} \sum_{\u c \in M(n,r)}
   \tau( \thu c \cdot \thu a \thu b^\vee )
      \cdot \thu c^\vee \]
for $\u a, \u b \in \Gamma$. But this is equal to 
$c_\chi^{-1} \thu a \thu b^\vee$ by Formula~(\ref{theformula}).
\end{proof}

\begin{Cor}
Let $\Gamma$ be a left cell and $\chi$ the corresponding irreducible
character of the left cell module $\LC^{(\Gamma)}$. Then the element
\[ e_\Gamma := \frac{1}{c_\chi} \sum_{\u a \in \Gamma}
   \thu a \thu a^\vee \]
is the central primitive idempotent of $K \Sa_q(n,r)$ corresponding
to the irreducible character $\chi$.
\end{Cor}

\begin{proof}
By Theorem~\ref{isotypbasis} $e_\Gamma$ lies in the isotypic component
corresponding to the character $\chi$ and is mapped to the identity
matrix in the corresponding matrix representation.
\end{proof}

\begin{Lemma}[Isomorphism of left cell modules and two-sided cells]
\label{isoLCLR}
Let $\Gamma$ and $\Gamma'$ be left cells. If $K\LC^{(\Gamma)}$ and
$K\LC^{(\Gamma')}$ are isomorphic $K\Sa_q(n,r)$-modules 
then\/ $\Gamma$ and\/ $\Gamma'$ lie in the same two-sided cell.
\end{Lemma}

\begin{proof}
Let $\chi$ be the irreducible character of the left cell module
$\LC^{(\Gamma)}$ and $\chi'$ that of $\LC^{(\Gamma')}$. The modules
$K\LC^{(\Gamma)}$ and $K\LC^{(\Gamma')}$ are isomorphic if and only if
$e_\Gamma \cdot e_{\Gamma'} = e_{\Gamma'} \cdot e_\Gamma \neq 0$ 
(and in this case $e_\Gamma = e_{\Gamma'}$).
Now assume this case. Then
\[ 0 \neq \frac{1}{c_\chi^2} \sum_{\u a \in \Gamma} \sum_{\u b \in \Gamma'}
   \thu a \thu a^\vee \thu b \thu b^\vee 
   = \frac{1}{c_\chi^2} \sum_{\u a \in \Gamma} \sum_{\u b \in \Gamma'}
   \thu b \thu b^\vee \thu a \thu a^\vee 
\]
and thus there is at least one pair $(\u a, \u b) \in \Gamma \times \Gamma'$
such that $\thu a^\vee \thu b \neq 0$. By Lemma~\ref{charLR}
this implies $\u a \le_{\sR} \u b$. Since $e_\Gamma$ and $e_{\Gamma'}$
commute, the same argument shows $\u b' \le_{\sR} \u a'$ for some
$\u a' \in \Gamma$ and $\u b' \in \Gamma'$. Thus, $\Gamma$ and
$\Gamma'$ lie in the same two-sided cell in that case.
\end{proof}

For what follows we need the following statement about
Iwahori-Hecke-Algebras of type A:

\begin{Theo}[Equal cell modules in the Iwahori-Hecke algebra]
\label{eqLCHa}
Let $\Ha$ be a generic Iwahori-Hecke-Algebra of type A as in 
Section~\ref{sec:qschur}. If $x \sim_{\sL} y$ and $z \sim_{\sL} w$ and
$x \sim_{\sR} z$ and $y \sim_{\sR} w$, then $C_x D_{y^{-1}} = C_z D_{w^{-1}}$.
In particular, we have
\[ g_{u,x,y} = \tau( C_u C_x D_{y^{-1}} ) = \tau( C_u C_z D_{w^{-1}})
  =g_{u,z,w} \]
for all $u \in W$.
\end{Theo}

\begin{proof}
This statement is already implicitly stated in \cite{KL79}. Namely,
it is shown there in the proof of Theorem 1.4 that the two left cell
modules defined by the left cell containing $x,y$ and the one containing
$z,w$ are isomorphic since all four lie in the same two-sided. The
exact statement there is that two $W$-graphs are isomorphic, which means
in particular that not only the two left cell modules are isomorphic,
but that even the matrix representations with respect to the bases
$\{ C_v \mid v \sim_{\sL} x\}$ and $\{ C_w \mid w \sim_{\sL} z \}$
are equal. But this exactly means, that
\[ \tau(D_{y^{-1}} C_u C_x) = \tau(D_{w^{-1}} C_u C_z) \]
for all $u \in W$ which we claim.
\end{proof}

Now we begin to use statements \textbf{Q1} to \textbf{Q14}:

\begin{Theo}[Equality of different left cell modules]
\label{eqLCmod}
Let $\Gamma, \Gamma'$ be left cells such that $K \LC^{(\Gamma)}$
and $K \LC^{(\Gamma')}$ are isomorphic $K\Sa_q(n,r)$-modules. Let
$\u d$ be the unique element in $\Gamma' \cap \D(n,r)$ (use \textbf{Q13})
and $\u c \sim_{\sL} \u d$ that is $\u c \in \Gamma'$. Then there are unique
$\u a, \u b \in \Gamma$ with $\u a \sim_{\sR} \u c$ and $\u b \sim_{\sR} \u d$
and we have 
$\thu a \thu b^\vee = \thu c \thu d^\vee$.
\end{Theo}

\begin{proof}
Let $\chi$ be the irreducible character of the left cell module
$\LC^{\Gamma'}$. We denote by $c_{\chi}$ the corresponding Schur
element. Since $\u c\sim_{\sL}\u d$, it follows from
Theorem~\ref{isotypbasis} that \[ \bas dc \bas cd =c_{\chi}\bas dd.
\] 
Therefore we have $\tau( \theta_{\u c}^\vee \bas cd \theta_{\u d}
) \neq 0$ and hence $\bas cd$ acts non-trivially on the module
$\LC^{(\Gamma')}$ (see Remark~\ref{apresformule}) and
thus also on the isomorphic module $\LC^{(\Gamma)}$.

This means that there is at least one pair 
$(\u a, \u b) \in \Gamma \times \Gamma$ such that
\[ \tau (\bas ba \cdot \bas cd) 
  = \tau( \theta_{\u a}^\vee \cdot \bas cd \cdot \theta_{\u b} )
  = \tau( \bas cd \cdot \bas ba ) \neq 0. \]
But then in particular $\thu a^\vee \thu c \neq 0$ and
thus $a \le_{\sR} c$ by Lemma~\ref{charLR}. 
Since $\Gamma$ and $\Gamma'$ lie in the same two-sided
cell by Lemma~\ref{isoLCLR}, we conclude $\u a \sim_{\sLR} \u c$ and thus by
\textbf{Q4} and \textbf{Q10} $\u a \sim_{\sR} \u c$. Analogously, we show
$\u b \sim_{\sR} \u d$. By Proposition~\ref{LReq} we conclude
that there is only one such pair $(\u a, \u b)$ since both are uniquely
defined by their membership in a left and a right cell.

We now show that $f_{\u e, \u a, \u b} = f_{\u e, \u c, \u d}$ for all
$\u e \in M(n,r)$ and thus $\bas ab = \bas cd$.
We have $co(\u a) = co(\u b)$ and 
$co(\u c) = co(\u d) = ro(\u d) = ro(\u b)$ and $ro(\u a) = ro(\u c)$
by Lemma~\ref{roco} and the fact that $\u d \in \D(n,r)$. Thus, if
$ro(\u e) \neq ro(\u b)$ or $co(\u e) \neq ro(\u a)$ then both
sides are zero by Lemma~\ref{scqschur}. Otherwise, we have
\[ f_{\u e, \u a, \u b} = h_{co(\u e)}^{-1} \cdot 
   g_{\sigma(\u e), \sigma(\u a), \sigma(\u b)} \qquad \mbox{and} \qquad
   f_{\u e, \u c, \u d} = h_{co(\u e)}^{-1} \cdot 
   g_{\sigma(\u e), \sigma(\u c), \sigma(\u d)} \]
and thus the equality $f_{\u e, \u a, \u b} = f_{\u e, \u c, \u d}$ follows
from \[
  \sigma(\u a) \sim_{\sL} \sigma(\u b) \sim_{\sR} \sigma(\u d) \sim_{\sL} \sigma(\u c)
 \sim_{\sR} \sigma(\u a)\] 
using Lemma~\ref{sigmale} and Theorem~\ref{eqLCHa}.
The non-degeneracy of $\tau$ now immediately implies $\bas ab = \bas cd$.
\end{proof}

With this we get the following result, for which we first need
one more piece of notation:

\begin{Def}[Schur elements of characters of left cell modules]
\label{def:schur}
Let $\u d \in \D(n,r)$ and $\Gamma$ the unique left cell with 
$\u d \in \Gamma$ (remember \textbf{Q13}). We denote the left cell
module $\LC^{(\Gamma)}$ by $\LC^{(\u d)}$ and the Schur
element corresponding to the irreducible character of\/ $\LC^{(\u d)}$
by $c_{\u d}$.
\end{Def}

\begin{Theo}[Wedderburn basis]
\label{wedderburn}
Let $\tau$ be an arbitrary non-degenerate symmetrising trace form on
$K\Sa_q(n,r)$.
The set
\[ \B := 
   \{ \cbas cd \mid \u c \in M(n,r), \u d \in \D(n,r), \u c \sim_{\sL} \u d \} \]
is a Wedderburn basis of $K \Sa_q(n,r)$. Two elements $\cbas cd$ and
$\cbas {c'}{d'}$ lie in the same isotypic component if and only if\/
$\LC^{(\u d)} \cong \LC^{(\u{d'})}$.

For $\cbas cd, \cbas{c'}{d'} \in \B$
we have the following equation:
\[ \cbas cd \cdot \cbas{c'}{d'} =
   \left\{ \begin{array}{ll}
     0 & \mbox{if } \LC^{(\u d)} \not\cong \LC^{(\u{d'})} \\
     0 & \mbox{if } \LC^{(\u d)} \cong \LC^{(\u{d'})} \mbox{ and }
                    \u d \not\sim_{\sR} \u{c'} \\
     \cbas{c''}{d'} & \mbox{if } \LC^{(\u d)} \cong \LC^{(\u{d'})} \mbox{ and }
                      \u d \sim_{\sR} \u{c'}
   \end{array}\right. \]
Here, $\u{c''}$ in the last case is the unique element with 
$\u{c''} \sim_{\sL} \u{d'}$ and $\u{c''} \sim_{\sR} \u c$ and the statement
contains the information that such a $\u{c''}$ in fact exists.
\end{Theo}

\begin{proof}
By Theorem~\ref{isotypbasis} the elements $\cbas cd$ and $\cbas{c'}{d'}$ 
both lie in an isotypic component. Thus,
if $\LC^{(\u d)} \not\cong \LC^{(\u{d'})}$ then clearly their product
is zero. 

Now assume that the left cell modules are isomorphic.
Let $\Gamma$ be an arbitrary left cell, such that $K\LC^{(\Gamma)}$ is
isomorphic to $K\LC^{(\u d)}$ and $K\LC^{(\u{d'})}$ and denote
the corresponding irreducible character by $\chi$. By Theorem~\ref{eqLCmod}
there are unique $\u a, \u b, \u{a'}, \u{b'} \in \Gamma$ with
\[ \u a \sim_{\sR} \u c \quad \mbox{and} \quad
   \u b \sim_{\sR} \u d \quad \mbox{and} \quad
   \u{a'} \sim_{\sR} \u{c'} \quad \mbox{and} \quad
   \u{b'} \sim_{\sR} \u{d'} \]
and we have $\bas ab = \bas cd$ and $\bas{a'}{b'} = \bas{c'}{d'}$.
Thus, Theorem~\ref{isotypbasis} implies that the product in the theorem
is $0$ if $\u b \neq \u{a'}$ and equal to $c_\chi^{-1} \bas a{b'}$
otherwise. We remark that if $\u d \sim_{\sR} \u{c'}$, then $\u{a'}
\sim_{\sR} \u{b}$ by transitivity. But using Proposition~\ref{LReq},
$\u{a'},\,\u{b}\in \Gamma$ implies $\u b = \u{a'}$. Hence $\u b =
\u{a'}$ if and only if $\u d \sim_{\sR} \u{c'}$ which proves case two in
the equation.

Finally, we assume also $\u d \sim_{\sR} \u{c'}$. Then, as $\u{c''}$ runs
through the left cell that contains $\u{d'}$, we can apply 
Theorem~\ref{eqLCmod} to each $\bas{c''}{d'}$ and the left cell $\Gamma$.
Since $\u b' \in \Gamma$ and $\u{b'} \sim_{\sR} \u{d'}$ we get that
\[ \{ \bas{c''}{d'} \mid \u{c''} \sim_{\sL} \u{d'} \}
 = \{ \bas{a''}{b'} \mid \u{a''} \in \Gamma \} \]
and both sets have cardinality $|\Gamma|$. Thus, there is a unique
$\u{c''}$ with $\bas{c''}{d'} = \bas a{b'}$ characterised by
$\u a \sim_{\sR} \u{c''} \sim_{\sL} \u{d'}$ and the theorem is proved.
\end{proof}

\begin{Cor}[Idempotents]
The elements $\cbas dd$ with $\u d \in \D(n,r)$ are pairwise orthogonal
primitive idempotents whose sum is the identity $1 \in \Sa_q(n,r)$. The central
primitive idempotent corresponding to an irreducible character $\chi$
of $K\Sa_q(n,r)$ is equal to
\[ \sum_{\renewcommand{\arraystretch}{0.5}
         \begin{array}{c} \scriptstyle \u d \in \D(n,r) \\
                          \scriptstyle \LC^{(\u d)} \mathrm{has~character~}\chi
         \end{array}} \cbas dd \]
\end{Cor}

\begin{proof}
This follows directly from Theorems~\ref{wedderburn}, \ref{eqLCmod} and 
\ref{isotypbasis}.
\end{proof}

\begin{Cor}[Left cell modules as submodules]
Let $\u d \in \D(n,r)$. Then the $A$-span
\[ \LL_{\u d} := \left< \bas cd \mid \u c \sim_{\sL} \u d \right>_A \]
is a left $\Sa_q(n,r)$-module by the multiplication in $K \Sa_q(n,r)$ 
that is isomorphic to the left cell
module $\LC^{(\u d)}$. In fact, the representing matrices with respect to 
the basis $(\bas cd)_{\u c \sim_{\sL} \u d}$ are equal to the representing
matrices coming from the left cell module $\LC^{(\u d)}$ with respect to
its standard basis.
\end{Cor}

\begin{proof}
Let $\Gamma$ be the left cell that contains $\u d$. Then by
Formula~(\ref{theformula}) we have for every $h \in \Sa_q(n,r)$: 
\[ h \thu
c = \sum_{\u{c'} \in M(n,r)} \tau( \thu{c'}^\vee \cdot h \thu c) \cdot
\thu{c'}. \] 
Moreover, for $\u a \in A$, there is $\alpha_{\u a}\in A$ such that
\[ h=\sum_{\u a\in M(n,r)}\alpha_{\u a}\theta_{\u a}. \]
Hence, for $\u c,\,\u{c'}\in M(n,r)$, we have $\tau( \thu{c'}^\vee \cdot
h \thu c)\in A$, because $\tau( \thu{c'}^\vee \cdot \thu a \thu c)\in A$
(see Remark~\ref{apresformule}). Multiplying this from the right with
$\thu d^\vee$ we get
\[ h \bas cd = \sum_{\u{c'} \in M(n,r)} \tau( h \bas
c{c'} )\cdot \bas{c'}d, \] 
where we only have to sum over $\u{c'} \in \Gamma$, since all the
summands are zero unless $\u d \le_{\sL} \u{c'} \le_{\sL} \u c$ by
Lemma~\ref{charLR}, which is equivalent to $\u{c'} \in \Gamma$. We then
deduce that $\LL_{\u d}$ is a left $\Sa_q(n,r)$-module. Moreover,
comparing with Remark~\ref{apresformule}, this shows the statement about
the representing matrices.
\end{proof}

\begin{Cor}
\label{Sqcontained}
The Schur algebra $\Sa_q(n,r)$ is contained in the $A$-span of the
Wedderburn basis $\B$:
\[ \Sa_q(n,r) \subseteq \left< \B \right>_A \]
\end{Cor}

\begin{proof}
Let $\Gamma_1, \ldots, \Gamma_n$ be left cells, such that the corresponding
left cell modules form a system of representatives for the isomorphism
types of simple left $K\Sa_q(n,r)$-modules. The mapping that maps
$h \in K\Sa_q(n,r)$
to its tuple of representing matrices in the cell modules
$\LC^{(\Gamma_1)}, \ldots, \LC^{(\Gamma_n)}$ with respect to their 
standard basis is an explicit isomorphism to a direct sum of
full matrix rings over $K$. In this isomorphism, the elements of $\B$
are mapped to a matrix unit, that is, to tuples of matrices, in which 
exactly one matrix is non-zero, and this matrix contains exactly one
non-zero coefficient equal to $1$. The elements of $\Sa_q(n,r)$ are
mapped to tuples of matrices with entries in $A$, since their representing
matrices on the cell modules have entries in $A$ (see the remark after
Theorem~\ref{SimpleCell}). Therefore, $\Sa_q(n,r)$ lies in the
$A$-span of $\B$.
\end{proof}

\begin{Prop}
\label{wedderdual}
Let $\tau$ be a  non-degenerate symmetrising trace form on
$K\Sa_q(n,r)$.
We denote by $\B$ the corresponding Wedderburn basis obtained in
Theorem~\ref{wedderburn}. Then, the dual basis of $\B$ relative to
$\tau$ is
\[ \B^{\vee}=\{\theta_{\u c}\theta_{\u d}^{\vee}\ |\ \u c\in M(n,r),\,\u
d\in\D(n,r),\,\u c\sim_{\sL}\u d\}. \]

\end{Prop}

\begin{proof}

Note first, that since $\tau$ is non-degenerate and $\B$ is a basis of
$K\Sa_q(n,r)$, there must be at least one element $c_{\u d'}^{-1}\,
\bas {c'}{d'} \in \B$
such that $\tau( c_{\u d}^{-1}\, \bas cd \cdot c_{\u d'}^{-1}\,\bas {c'}{d'} )$ 
is non-zero. Since $c_{\u d'}\neq 0$, we have in particular 
$\tau( c_{\u d}^{-1}\, \bas cd \bas {c'}{d'} ) \neq 0$.
We try
to find out, which element $\bas {c'}{d'}$ this can be:

By Theorem~\ref{wedderburn}, the value 
$\tau(c_{\u d}^{-1}\, \bas cd\, \bas {c'}{d'} )$ is equal to zero, if
$\LC^{(\u d)} \not\cong \LC^{(d')}$ or $\u d \not\sim_{\sR} \u {c'}$.
If however $\LC^{(\u d)} \cong \LC^{(d')}$ and $\u d \sim_{\sR} \u{c'}$,
then it is equal to $\tau(\bas{c''}{d'} )$ where $\u{c''}$ is uniquely
defined by $\u{c''} \sim_{\sL} \u{d'}$ and $\u{c''} \sim_{\sR} \u c$.
If $\u{c''} \neq \u{d'}$, then this value is also equal to $0$ because
of the original definition of $\{\theta_{\u a}^\vee \mid \u a \in M(n,r)\}$.
If however $\u{c''} = \u{d'}$ we can show that $\u{c'} = \u c^t$
using Proposition~\ref{LReq}: Namely, we have $\u{c'} \sim_{\sL} \u{d'}
= \u{c''} \sim_{\sR} \u c$ and thus $\u{c'} \sim_{\sL} \u c^t$ by
transposition. Further, we have $\u{c'} \sim_{\sR} \u d \sim_{\sL} \u c$
and thus again by transposition $\u{c'} \sim_{\sR} \u c^t$. Thus,
$\u {c'}$ and $\u c^t$ are both left and right equivalent and therefore
equal.

Thus, we deduce that
\[ \tau(c_{\u d}^{-1}\, \bas cd \cdot \bas {c'}{d'} ) = \delta_{\u{c'},\u c^t} \]
for all $\u c \in M(n,r)$ and $\u d \in \D(n,r)$ with $\u c \sim_{\sL} \u
d$, and all $\u c' \in M(n,r)$ and $\u {d'} \in \D(n,r)$ with
$\u {c'} \sim_{\sL} \u {d'}$.
\end{proof}

\begin{remark}
Note that as a byproduct we have proved the following result: If 
$\u c \in M(n,r)$ and $\u d \in \D(n,r)$ with $\u c \sim_{\sL} \u
d$, and $\u {d'} \in \D(n,r)$ with $\u c^t \sim_{\sL} \u {d'}$,
then $\LC^{(\u d)} \cong \LC^{(\u {d'})}$.
\end{remark}

We now talk about $A$-sublattices of $K\Sa_q(n,r)$.

\begin{DefProp}[$A$-sublattices of $K\Sa_q(n,r)$ and their duals]
By an \textbf{$A$-lattice in $K\Sa_q(n,r)$} we mean
an $A$-free $A$-submodule that contains a $K$-basis of $K\Sa_q(n,r)$.
Let $L \subseteq K\Sa_q(n,r)$ be an $A$-lattice. Then we set
\[ L^\vee := \{ h \in K\Sa_q(n,r) \mid \tau(h x) \in A \mbox{ for all }
   x \in L \} \]
and call it the \textbf{dual lattice of $L$}. Since $\tau$ is non-degenerate,
$L^\vee$ is again an $A$-lattice in $K \Sa_q(n,r)$, namely, if $(b_{\u
a})_{\u a \in M(n,r)}$ is an $A$-basis of $L$, then the dual basis
$(b_{\u a}^\vee)_{\u a \in M(n,r)}$ is an $A$-basis of $L^\vee$. 
Clearly, if $L \subseteq N$ are two
$A$-lattices in $K\Sa_q(n,r)$, then $N^\vee \subseteq L^\vee$.

Note that we do not require an $A$-lattice to be an $A$-algebra!
\hfill $\square$
\end{DefProp}

\begin{Prop}[The dual is an $\Sa_q(n,r)$-module]
We have
$\Sa_q(n,r) \cdot \Sa_q(n,r)^\vee \subseteq \Sa_q(n,r)^\vee$.
\end{Prop}

\begin{proof}
Fix $h \in \Sa_q(n,r)$ and $k \in \Sa_q(n,r)^\vee$. 
We have to show that $hk \in \Sa_q(n,r)^\vee$. However, for every 
$x \in \Sa_q(n,r)$ holds $\tau(hkx)=\tau(kxh)$.
Since $xh \in \Sa_q(n,r)$ (because $\Sa_q(n,r)$ is an algebra), 
and $k \in \Sa_q(n,r)^\vee$ we get $\tau(kxh)\in A$.
\end{proof}

For the rest of this section we let 
$\tau = \sum_{\chi \in \Irr(K\Sa_q(n,r))} \chi$, that is,
we choose $\tau$ such that all Schur elements are equal to $1$.

\begin{Prop}[The Wedderburn-basis is self-dual]
\label{wedderselfdual}
Let $\tau = \sum_{\chi \in \Irr(K\Sa_q(n,r))} \chi$.
Then \[ \left<\B\right>_A^\vee = \left<\B\right>_A\] for the Wedderburn basis
$\B$ from Theorem~\ref{wedderburn}.
\end{Prop}

\begin{proof}
Since $\tau$ is the sum of the irreducible characters, all Schur
elements $c_\chi$ are equal to one. It is then a direct consequence of
Proposition~\ref{wedderdual}.
\end{proof}

\begin{Cor}[The dual of $\Sa_q(n,r)$] From 
Lemma~\ref{Sqcontained} and Proposition~\ref{wedderselfdual} follows
\[ \left< \B \right>_A \subseteq \Sa_q(n,r)^\vee \]
\end{Cor}

\begin{proof} Dualising reverses inclusion. \end{proof}

\section{The asymptotic algebra and the Du-Lusztig
homomorphism}\label{sec:asym}

In this section we briefly recall the definition of the asymptotic
algebra $\J(n,r)$ for the $q$-Schur algebra $\Sa_q(n,r)$ and of the
Du-Lusztig homomorphism $\Phi$ from $\Sa_q(n,r)$ to $\J(n,r)$. We then
show that this algebra is isomorphic to the algebra $\left< \B
\right>_A$ spanned by our Wedderburn basis $\B$ and that the Du-Lusztig
homomorphism can be interpreted as the inclusion of $\Sa_q(n,r)$ into
$\left< \B \right>_A$.

\begin{Def}[The asymptotic algebra $\J(n,r)$]
Let $\J(n,r)$ be the free abelian group with basis $\{t_{\u a} \mid 
\u a \in M(n,r)\}$. We define a multiplication on $\J(n,r)$ by setting
\[ t_{\u a} t_{\u b} = \sum_{\u c \in M(n,r)} \gamma_{\u a,\u b,\u c^t} 
   \cdot t_{\u c}. \]
We set $\D(n,r)_\lambda := \D(n,r) \cap M_{\lambda,\lambda}$.
Following Du, we denote the extension of scalars of $\J(n,r)$ to $A$
by $\J(n,r)_A$.
\end{Def}

\begin{Lemma}[{See \cite[(2.2.1)]{DuAsymp}}]
The $\Z$-algebra $\J(n,r)$ is associative with the identity element
\[ \sum_{\u d \in \D(n,r)} t_{\u d}. \]
\end{Lemma}

\begin{Theo}[{The Du-Lusztig homomorphism $\Phi$, see \cite[(2.3]{DuAsymp}}]
\label{DuLusHom}
The $A$-linear map $\Phi : \Sa_q(n,r) \to \J(n,r)_A$ defined by
\[ \Phi( \thu a ) := 
   \sum_{ \renewcommand{\arraystretch}{0.5}\begin{array}{c}
     \scriptstyle \u b \in M(n,r) \\
     \scriptstyle \u d \in \D(n,r)_\mu \\
     \scriptstyle \a(\u d) = \a(\u b)
     \end{array} } f_{\u a,\u d,\u b} \cdot t_{\u b} 
 = \sum_{ \renewcommand{\arraystretch}{0.5}\begin{array}{c}
     \scriptstyle \u b \in M(n,r) \\
     \scriptstyle \u d \in \D(n,r) \\
     \scriptstyle \u d \sim_{\sL} \u b
     \end{array} } f_{\u a,\u d,\u b} \cdot t_{\u b}, \qquad
 \mbox{where } \mu = co(\u a) \]
is an algebra homomorphism and becomes an isomorphism 
$K\Sa_q(n,r) \to \J(n,r)_K$ when tensored with
the field of fractions $K$ of $A$.
\end{Theo}

\begin{proof} See \cite[2.3]{DuAsymp}. The latter equation holds,
since $f_{\u a,\u b,\u d} = 0$ unless $\u d \le_{\sL} \u b$, and
\textbf{Q9} implies $\u d \sim_{\sL} \u b$ in this case. Also we can
safely sum over all of $\D(n,r)$ neglecting the index $\mu$, since
all elements $\u d \in \D(n,r)$ fulfill $ro(\u d) = co(\u d)$ by
definition (see Definition~\ref{defD} and the remark there) and
$f_{\u a,\u d,\u b} = 0$ unless $co(\u a) = ro(\u d)$ anyway.
\end{proof}

We can now present our main theorem, which links our Wedderburn basis
$\B$ to the asymptotic algebra:

\begin{Theo}[Preimage of the $t$-basis under the Du-Lusztig homomorphism]
\label{preimages}
Let $\tau$ be an arbitrary non-degenerate symmetrising trace form. All dual
bases in the following are meant with respect to $\tau$.

With the above notation we have
\[ \Phi( \cbas cd ) = t_{\u c} \qquad \mbox{for all } \u c \in M(n,r). \]
\end{Theo}

\begin{proof}
The rightmost sum in Theorem~\ref{DuLusHom} has the advantage that
it provides a formula for the image of an arbitrary
element $h \in K \Sa_q(n,r)$ under the Du-Lusztig homomorphism,
since it is obviously $K$-linear in $\thu a$:
\[ \Phi( h ) = \sum_{ \renewcommand{\arraystretch}{0.5}\begin{array}{c}
     \scriptstyle \u b \in M(n,r) \\
     \scriptstyle \u {d'} \in \D(n,r) \\
     \scriptstyle \u {d'} \sim_{\sL} \u b
     \end{array} } \tau( h \cdot \bas{d'}b ) \cdot t_{\u b} \]
(recall $\tau( \thu a \thu{d'} \thu b^\vee ) = f_{\u a,\u{d'},\u b}$).
But now we can immediately set $h := \cbas cd$ for some $\u c \in M(n,r)$
and $\u d \in \D(n,r)$ with $\u c \sim_{\sL} \u d$. The value
$\tau( \cbas cd \cdot \bas{d'}b )$ is zero (see Lemma~\ref{charLR}) unless
$\u b \le_{\sR} \u c \sim_{\sL} \u d \le_{\sR} \u {d'} \sim_{\sL} \u b$ and
this implies $\u b \sim_{\sR} c$ and $\u{d'} \sim_{\sR} \u d$
using \textbf{Q4} and \textbf{Q10}. But this means $\u{d'} = \u d$ by
\textbf{Q13} and the definition of $\sim_{\sR}$ 
and thus $\u b = \u c$ because
of Lemma~\ref{LReq}. Thus, in the sum there is only one non-zero summand,
which is $\tau( \cbas cd \cdot \bas dc ) t_{\u c}$. Now everything is in
a single left cell such that we can use Theorem~\ref{isotypbasis}
to get
\[ \tau( \cbas cd \cdot \bas dc ) \cdot t_{\u c}
 = \tau( \bas cc ) \cdot t_{\u c} = t_{\u c} \]
as claimed.
\end{proof}

We can summarise our results in the following way: 

\begin{Theo}[New interpretation of the Du-Lusztig homomorphism]
\label{newinter}
Let $\tau$ be an arbitrary non-degenerate symmetrising trace form on
$K\Sa_q(n,r)$. We define the set $\B$ as in Theorem~\ref{wedderburn}
and we set
\[ \J_{\tau}= \left<\B\right>_A. \]
The following diagram
commutes and all unmarked arrows are identities or natural inclusions:
\[ \xymatrix{
   \Sa_q(n,r) \ar[r] \ar@{=}[d] &
   \J_{\tau} \ar[r] \ar[d]_\Phi^\cong &
   K\Sa_q(n,r) \ar[d]_\Phi^\cong \\
   \Sa_q(n,r) \ar[r]^\Phi &
   \J(n,r)_A \ar[r] &
   \J(n,r)_K } \]
Thus, the asymptotic algebra $\J(n,r)_A$ is nothing but the $A$-span of
our Wedderburn basis and the Du-Lusztig homomorphism $\Phi$ can simply be
interpreted as the inclusion of $\Sa_q(n,r)$ into $\left< \B \right>_A$.
Furthermore, our results directly and explicitly show that 
$\left< \B \right>_A$ is isomorphic as an $A$-algebra to a direct sum of
full matrix rings over $A$.
\end{Theo}

\section{A criterion for James'
conjecture}\label{partjames}

In this section we show how our results provide an equivalent formulation
of a conjecture about the representation theory of specialisations of the
$q$-Schur algebra. We first recall the conjecture.

The construction of the Iwahori-Hecke algebra of type A and of the
$q$-Schur algebra as in Section~\ref{sec:qschur} together with their
Kazhdan-Lusztig bases can be carried out over an arbitrary integral
domain $R$ with quotient field $k$ and with an arbitrary invertible
parameter $q \in R$ having a square root in that domain. We denote the
resulting algebra by $\Sa_q(n,r)_R$ and its extension of scalars to $k$
by $\Sa_q(n,r)_k$.

The case of the Laurent polynomial ring $A =
\Z[v,v^{-1}]$ and $q=v^2$ is called the ``generic'' case, since for
every other choice $(R,q)$ there is a ring homomorphism
$\varphi: \Z[v,v^{-1}] \to R$ mapping $v^2$ to $q \in R$, which
induces a ring homomorphism 
$\Sa_{v^2}(n,r)_A \to \Sa_q(n,r)_R \subseteq \Sa_q(n,r)_k$. This
is called a ``specialisation''.

It is known, that $\Sa_q(n,r)_k$ is semisimple unless $q$ is an $e$-th
root of unity. If $q$ is a root of unity, then there is a decomposition
matrix, which records the multiplicities of the simple modules in 
the so-called ``standard modules''. For the case that $k$ has characteristic
zero, recent work by Lascoux, Leclerc and Thibon, and Varagnolo and Vasserot
yields a complete determination of these decomposition matrices
(see \cite{varvas}, \cite{geckbourb} and the references there). However,
the case of positive characteristic is still open.

James' conjecture is a statement about this modular case. Roughly speaking,
it asserts that if $k$ is a field of characteristic $\ell$ and the 
multiplicative order $e$ of the parameter $q \in k$
is greater than $r$, then the decomposition matrix of $\Sa_q(n,r)_k$
does not depend on the particular value of $\ell$ but only on $e$.

We now want to make this statement more precise. Both the simple
modules and the standard modules have a labelling by the set $\Lambda(n,r)$.
Let $V^\lambda_{k,q}$ denote the standard module and $M^\lambda_{k,q}$ the 
simple module of $\Sa_q(n,r)_k$ corresponding to $\lambda$ and $\mu$ 
respectively. Then the decomposition matrix for $\Sa_q(n,r)_k$ consists 
of the numbers
\[ d^{k,q}_{\lambda,\mu} := 
   \mbox{multiplicity of $M^\mu_{k,q}$ in $V^\lambda_{k,q}$}.
\]

\begin{Conj}[{James, see \cite[\S4]{JamesConj} and \cite[\S3]{geckbourb}}]
If $\ell > r$ and $e$ is the multiplicative order of $q \in k$, then 
$d^{k,q}_{\lambda,\mu} = d^{\Q(\zeta_e),\zeta_e}_{\lambda,\mu}$ for
all $\lambda,\mu \in \Lambda(n,r)$, where
$\zeta_e$ is a complex primitive $e$-th root of unity.
\end{Conj}

Meinolf Geck has shown in \cite[Theorem 1.2]{geckqschurjames} that this
statement is equivalent to the fact, that for $\ell > r$, the rank of the
Du-Lusztig homomorphism $\Phi : \Sa_q(n,r)_k \to \J(n,r)_k$ with respect
to the two bases $(\theta_{\u a})_{\u a \in M(n,r)}$ and 
$(t_{\u a})_{\u a \in M(n,r)}$
respectively only depends on the multiplicative order $e$ of $q \in k$
and not on the characteristic $\ell$ of $k$. 

In view of our Theorem~\ref{newinter} this immediately implies:

\begin{Theo}[An equivalent formulation of James' conjecture]
\label{JamesMeinolf}
Let $\{\theta_{\u a} \mid \u a \in M(n,r)\}$ be the
Du-Kazhdan-Lusztig-basis of $\Sa_q(n,r)$ and let $\tau$ be a
non degenerate symmetrising trace form for 
$K\Sa_q(n,r)$. Let $\{ \theta_{\u a}^\vee
\mid \u a \in M(n,r)\}$ be the dual basis of  
$\{\theta_{\u a} \mid \u a \in M(n,r)\}$ with respect to $\tau$.
Let $\B$ be the basis defined in Theorem~\ref{wedderburn}.
Let $s := |M(n,r)|$ and 
$M = (m_{\u a,\u b})_{\u a, \u b \in M(n,r)} 
\in A^{s \times s}$
be the matrix, for which
\[ \theta_{\u a} = \sum_{\u c\, \in M(n,r)} m_{\u a, \u c} \cdot \cbas cd \]
with $\cbas cd \in \B$ holds for all $\u a \in M(n,r)$.

Let $\ell_1, \ell_2$ be two primes and
$\varphi_1 : \Z[v,v^{-1}] \to \mathbb{F}_{\ell_1}$ and
$\varphi_2 : \Z[v,v^{-1}] \to \mathbb{F}_{\ell_2}$ two ring homomorphisms, such
that the multiplicative orders of $\varphi_1(v^2)$ and $\varphi_2(v^2)$
are equal. Denote by $\varphi_i(M)$ the matrix in 
$\mathbb{F}_{\ell_i}^{s \times s}$
that one gets by applying the ring homomorphism $\varphi_i$ to every
entry of $M$. 

Then James' conjecture is equivalent to the fact, that for $\ell_1,
\ell_2 > r$ the ranks of $\varphi_1(M)$ and of $\varphi_2(M)$ are equal.
\end{Theo}

Let $\tau$ be a non-degenerate symmetrising trace form on $K\Sa_q(n,r)$.
We denote by $\{\theta_{\u a} \mid \u a \in M(n,r)\}$ the
Du-Kazhdan-Lusztig-basis of $\Sa_q(n,r)$ and by $\{\theta_{\u 
a}^{\vee} \mid \u a \in M(n,r)\}$ its dual basis relative to $\tau$.
As above, we denote by $\B$ the Wedderburn basis obtained in
Theorem~\ref{wedderburn}. 
Moreover, we denote by $M = (m_{\u a, \u b})_{\u a, \u b \in M(n,r)}$ 
the change of basis matrix 
from $\{\theta_{\u a} \mid \u a \in M(n,r)\}$ to $\B$ as above and by 
$P_{\tau} = (p_{\u a, \u b})_{\u a, \u b \in M(n,r)}$ the change of basis matrix
from  $\{\theta_{\u a} \mid \u a
\in M(n,r)\}$ to  $\{\theta_{\u a}^{\vee} \mid \u a \in M(n,r)\}$,
that is:
\[ \theta_{\u a} = \sum_{\u b \in M(n,r)} p_{\u a, \u b} \cdot
\theta^\vee_{\u b} \]
for all $\u a \in M(n,r)$.
Formula~(\ref{theformula}) implies that 
\[ P_{\tau}=\left(\tau(\thu a\thu b)\right)_{\u a,\u b\in
M(n,r)}\quad\textrm{and}\quad 
P_{\tau}^{-1}=\left(\tau(\thu a^{\vee}\thu b^{\vee})\right)_{\u a,\u b\in
M(n,r)}. \]

\begin{Lemma}
\label{diag}
With the above notation, the matrix
\[ D=M^T P_{\tau}^{-1} M \]
is monomial and its entries are the Schur elements $c_{\u d}$ associated 
to $\u d\in\D(n,r)$ as in Definition~\ref{def:schur}.
\end{Lemma}
\begin{proof}
The matrix $M^T$ is the change of basis matrix from $\B^{\vee}$ to 
$\{\theta_{\u a}^{\vee} \mid \u a \in M(n,r)\}$ and thus
the matrix $D$ is the change of basis matrix from
$\B^{\vee}$ to $\B$, that is:
\[ \theta_{\u c} \theta^\vee_{\u d} = 
\sum_{\u {c'} \in M(n,r)} d_{\u c,\u {c'}} c_{\u d'}^{-1} \theta_{\u
c'} \theta^\vee_{\u d'} \]
for all $\theta_{\u c} \theta^\vee_{\u d} \in \B^{\vee}$.
Using Proposition~\ref{wedderdual}, the result follows. 
\end{proof}

\begin{Prop}[A criterion for James' conjecture]
\label{reformulation}
Let $\tau$ be a non-degenerate symmetrising
trace form on $K\Sa_q(n,r)$. Let $\varphi_e : A \to \Z[\zeta_{2e}], v
\mapsto \zeta_{2e}$ be a specialisation to characteristic $0$ where
$v^2$ is mapped to a primitive $e$-th root of unity in a cyclotomic
field and $\varphi_\ell : A \to \mathbb{F}_\ell$ is a 
second specialisation to characteristic $\ell$ 
such that there is a ring homomorphism
$\varphi^e_\ell : \mathbb{Z}[\zeta_{2e}] \to \mathbb{F}_\ell$ with 
$\varphi_\ell =
\varphi^e_\ell \circ \varphi_e$.
We suppose that
$\ell > r$ and the following hypotheses on $\tau$:
\begin{itemize}
\item The Schur elements $c_{\u d}$ for $\u d\in \D(n,r)$ lie in $A$.
\item The coefficients of the matrix $P_{\tau}^{-1}$ lie in $A$.
\item Let $a$ be the number of Schur elements $c_{\u d}$ for $\u d \in \D(n,r)$
    that do not vanish under $\varphi_e$ and $b$ the number of Schur
    elements that do not vanish under $\varphi_{\ell}$.
    The numbers $a$ and $b$ are both equal to the rank over
    $\mathbb{Q}(\zeta_{2e})$ of the 
    matrix $\varphi_e(M)$ for $M$ from above.
\end{itemize}
Note that we denote with the notation $\varphi_e(M)$ the matrix one gets
from $M$ by applying the ring homomorphism $\varphi_e$ on every entry.

If $\tau$ can be found fulfilling all these hypotheses, then James'
conjecture holds for all $\ell > r$ for which $\varphi_\ell$ as above
exist.
\end{Prop}
\begin{proof}
We denote by $M$ the change of basis matrix
from $\{\theta_{\u a} \mid \u a \in M(n,r)\}$ to $\B$ as above.
Then Lemma~\ref{diag} asserts that 
\[ D=M^{T} P_{\tau}^{-1} M.\] 
Thanks to Theorem~\ref{wedderburn}, the coefficients of the matrix
$M$ lie in $A$. By hypothesis, the matrix $P_{\tau}^{-1}$ has coefficients in
$A$. By Lemma~\ref{diag} and the first hypothesis the entries of $D$ are
also in $A$.

Since the matrices $D$, $M$, $M^T$, and $P_{\tau}^{-1}$ have
coefficients in $A$, the matrices $\varphi_e(D)$, $\varphi_e(M)$,
$\varphi_{\ell}(D)$, $\varphi_{\ell}(M)$, $\varphi_{\ell}(M^T)$ and
$\varphi_{\ell}(P_{\tau}^{-1})$ are well-defined.
We then have the following equality
\[ \varphi_\ell(D) = \varphi_\ell(M^T) \cdot \varphi_\ell(P_{\tau}^{-1}) 
\cdot
   \varphi_\ell(M), \]
implying that $\operatorname{rk}_{\mathbb{F}_{\ell}}(\varphi_{\ell}(D))\leq
\operatorname{rk}_{\mathbb{F}_{\ell}}(\varphi_{\ell}(M))$.
Moreover we have $\varphi_{\ell}(M)=\varphi_{\ell}^e(\varphi_e(M))$.
Since $\varphi_{\ell}^e$ is a ring homomorphism, we deduce that
$$\operatorname{rk}_{\mathbb{F}_{\ell}}(\varphi_{\ell}(M))\leq\operatorname{rk}_{\Q(\zeta_{2e})}(\varphi_e(M)).
$$
Since $D$ is a monomial matrix containing only the Schur
elements as non-zero entries, the numbers $a$ and $b$ from the
hypotheses are the ranks of $\varphi_e(D)$ and $\varphi_\ell(D)$
respectively. 
However, if as in the last hypothesis the ranks of 
$\varphi_e(M)$ and $\varphi_\ell(D)$ are equal,
then it follows that
$\operatorname{rk}_{\mathbb{F}_{\ell}}(\varphi_\ell(M))\leq
\operatorname{rk}_{\mathbb{F}_{\ell}}(\varphi_\ell(D))$.
We then deduce that
$$\operatorname{rk}_{\mathbb{F}_{\ell}}(\varphi_\ell(M))=
\operatorname{rk}_{\mathbb{F}_{\ell}}(\varphi_\ell(D)),$$
and the result now follows from Theorem~\ref{JamesMeinolf}.
\end{proof}

\begin{remark}
To prove James' conjecture it is enough to find a symmetrising
trace form $\tau$ on $K\Sa_q(n,r)$ such that the hypotheses
of Proposition~\ref{reformulation} are satisfied. We notice
that the assumption on $P_{\tau}$ in the statement of
Proposition~\ref{reformulation} is ``generic'' in the sense that this
property only depending on the ``generic'' $q$-Schur algebra, but not on
specialisations over finite fields.
\end{remark}
\begin{remark}
We can replace the second assumption of
Proposition~\ref{reformulation} by the fact that the matrix
$P_{\tau}^{-1}M$ (or $M^TP_{\tau}^{-1}$) has its coefficients in $A$.
\end{remark}
\begin{remark}For the usual trace form $\tau$ on Hecke algebras of type
$A$, we note that the assumptions of Proposition~\ref{reformulation}
hold. Then using~\cite{klwedder}, we can prove in a way similar
to the one of the proof of Proposition~\ref{reformulation},
that the rank of the Lusztig homomorphim (specialized in a
finite field $\mathbb{F}_{\ell}$ by $\varphi_{\ell}:A\rightarrow
\mathbb{F}_{\ell}$ mapping $v^2$ to an element $q\in\mathbb{F}_{\ell}$
with multiplicative order $e$ as above) does not depend on $\ell$.
However as noted by Geck in~\cite{geckqschurjames} an analogue result as
Theorem~\ref{JamesMeinolf} in Hecke algebras does not imply the Hecke
algebras James' conjecture.
\end{remark}

\bibliographystyle{plain}
\bibliography{klqschur}

\end{document}